\documentclass{amsart}

\newtheorem{theorem}{Theorem}[section]
\newtheorem{lemma}[theorem]{Lemma}

\theoremstyle{definition}

\theoremstyle{remark}

\numberwithin{equation}{section}

\newcommand{\refth}[1]{Theorem~\ref{#1}}
\newcommand{\reflm}[1]{Lemma~\ref{#1}}

\newcommand{\refsec}[1]{Section~\ref{#1}}

\newcommand{\refeq}[1]{(\ref{#1})}

\usepackage{amsfonts}
\usepackage{amssymb}
\usepackage{graphicx}
\usepackage{enumerate}
\usepackage{amsmath}

\newcommand{\beeq}[1]{\begin{equation} \label{#1}}
\newcommand{\eeq}{\end{equation}}
\newcommand{\beeqs}{\begin{eqnarray*}}
\newcommand{\eeqs}{\end{eqnarray*}}
\renewcommand{\(}{\begin{eqnarray*}}
\renewcommand{\)}{\end{eqnarray*}}
\newcommand{\beeqn}{\begin{eqnarray}}
\newcommand{\eeqn}{\end{eqnarray}}

\newcommand{\nexteqline}{\\ &=&}
\newcommand{\eqand}{\mbox{~~~and~~~}}

\renewcommand{\quad}{\hspace*{3mm}}
\renewcommand{\qquad}{\hspace*{5mm}}

\newcommand{\lp}{\left(  }
\newcommand{\rp}{\right) }
\newcommand{\lb}{\left\{  }
\newcommand{\rb}{\right\} }
\newcommand{\lbr}{\left[  }
\newcommand{\rbr}{\right] }

\newcommand{\lc}{\left\lceil}
\newcommand{\rc}{\right\rceil}

\newfont{\Bb}{msbm8 scaled\magstep1}

\newcommand{\Z}{\mbox{\Bb Z}}

\newcommand{\R}{\mbox{\Bb R}}

\newcommand{\ep}{\epsilon}

\begin{document}

\title{Asymptotic Improvement of the Sunflower Bound}


\author{Junichiro Fukuyama}
\address{Applied Research Laboratory\\
The Pennsylvania State University\\
PA 16802, USA}
\curraddr{}
\email{jxf140@psu.edu}
\thanks{}


\subjclass[2010]{05D05:Extremal Set Theory (Primary)}

\keywords{Sunflower Lemma, Sunflower Conjecture, $\Delta$-System}

\date{}

\dedicatory{}

\begin{abstract}
A {\em sunflower with a core $Y$} is a family ${\mathcal B}$ of sets such that $U \cap U' = Y$ for each two different elements $U$ and $U'$ in ${\mathcal B}$. The well-known sunflower lemma states that a given family ${\mathcal F}$ of sets, each of cardinality at most $s$, includes a sunflower of cardinality $k$ if $|{\mathcal F}|> (k-1)^s s!$. Since Erd\"os and Rado proved it in 1960, it has not been known for more than half a century whether the sunflower bound $(k-1)^s s!$ can be improved asymptotically for any $k$ and $s$. It is conjectured that it can be reduced to $c_k^s$ for some real number $c_k>0$ depending only on $k$, which is called the {\em sunflower conjecture}. This paper shows that the general sunflower bound can be indeed reduced by an exponential factor: 
We prove that ${\mathcal F}$ includes a sunflower of cardinality $k$ if 
\(
&&
|{\mathcal F}| \ge \lp \sqrt{10} -2 \rp^2 \lbr k \cdot \min \lp \frac{1}{\sqrt{10} -2}, \frac{c}{\log \min (k, s)} \rp \rbr^s s!,
\)
for a constant $c>0$, and any $k \ge 2$ and $s \ge 2$. For instance, whenever $k \ge s^\ep$ for a given constant $\ep \in (0,1)$, the sunflower bound is reduced from $(k-1)^s s!$ to $(k-1)^s s! \cdot \lbr O \lp \frac{1}{\log s}\rp \rbr^s$, achieving the reduction ratio of $\lbr O \lp \frac{1}{\log s}\rp\rbr^s$. Also any ${\mathcal F}$ of cardinality at least $\lp \sqrt{10} - 2 \rp^2 \lp \frac{k}{\sqrt{10}-2} \rp^s s!$ includes a sunflower of cardinality $k$, where $\frac{1}{\sqrt{10} -2}=0.8603796 \ldots$.
Our result demonstrates that the sunflower bound can be improved by a factor of less than a small constant to the power $s$, giving hope for further update.
\end{abstract}

\maketitle

\section{Introduction}

A set means a subset of a given universal set $X$. Denote by ${\mathcal F}$ a family of sets, and by ${\mathcal B}$ its sub-family. For a set $Y \subset X$, a {\em sunflower with a core $Y$} is a family ${\mathcal B}$ of sets such that $U \cap U' = Y$ for each two different elements $U$ and $U'$ in ${\mathcal B}$. Equivalently, ${\mathcal B}$ is a sunflower if $U \cap U' = \bigcap_{V \in {\mathcal B}} V$ for any $U, U'  \in {\mathcal B}$ such that $U \ne U'$.
A sunflower of cardinality $k$ is called {\em $k$-sunflower} for short.  A {\em constant} is a fixed positive real number depending on no variable.

The sunflower lemma shown by Erd\"os and Rado \cite{OriginalSF} states that:

\begin{lemma} 
A family ${\mathcal F}$ of sets, each of cardinality at most $s$, includes a $k$-sunflower if $|{\mathcal F}|> (k-1)^s s!$. 
\end{lemma}

Since its proof was given in 1960, it has not been known whether the sunflower bound $(k-1)^s s!$ can be asymptotically improved for any $k$ and $s$, despite its usefulness in combinatorics and various applications \cite{Fu91, Jukna}. It is conjectured that the bound can be reduced to $c_k^s$ for a real number $c_k>0$ only depending on $k$, which is called {\em the sunflower conjecture}. The results known so far related to this topic include:

\begin{enumerate} [-]
\item Kostochka \cite{Ko97} showed that 
the sunflower bound for $k=3$ is reduced from $2^s s!$ to 
$c s! \lp \frac{\log \log \log s}{\log \log s} \rp^s$ for a constant $c$. The case $k=3$ of the sunflower conjecture is especially emphasized by Erd\"os \cite{E81}, which other researchers also believe includes some critical difficulty. 
\item It has also been shown \cite{KRT99} that ${\mathcal F}$ of cardinality greater than $k^s\lp 1+ c_s k^{-2^{-s}}\rp$ includes a $k$-sunflower for some $c_s \in \R^+$ depending only on $s$. 
\item With the sunflower bound $(k-1)^s s!$, Razborov proved an exponential lower bound on the monotone circuit complexity of the clique problem \cite{Raz}. Alon and Boppana strengthened the bound \cite{AB87}
by relaxing the condition to be a sunflower from $U \cap U' = \bigcap_{V \in {\mathcal B}} V$  to $U \cap U'  \supset \bigcap_{V \in {\mathcal B}} V$ for all $U, U' \in {\mathcal F}, U \ne U'$. 
\item \cite{ASU13} discusses the sunflower conjecture and its variants in relation to fast matrix multiplication algorithms. Especially, it is shown in the paper that if the sunflower conjecture is true, the Coppersmith-Winograd conjecture implying a faster matrix multiplication algorithm \cite{CW90} does not hold. 
\end{enumerate}

In this paper we show that the general sunflower bound can be indeed improved by an exponential factor. We prove the following theorem.

\begin{theorem} \label{Main}
A family ${\mathcal F}$ of sets, each of cardinality at most $s$, includes a $k$-sunflower if
\[
|{\mathcal F}| \ge \lp \sqrt{10}-2 \rp^2 \lbr k \cdot \min \lp \frac{1}{\sqrt{10} -2}, \frac{c}{\log \min (k, s)} \rp \rbr^s s!,
\]
for a constant $c$ and any integers $k \ge 2$ and $s \ge 2$. 
\end{theorem}

This improves the sunflower bound by the factor of $\lbr O \lp \frac{1}{\log s} \rp \rbr^s$ whenever $k$ exceeds $s^\ep$ for a given constant $\ep \in (0, 1)$. Also any ${\mathcal F}$ of cardinality at least $\lp \sqrt{10} - 2 \rp^2 \lp \frac{k}{\sqrt{10} -2}\rp^s s!$ includes a $k$-sunflower.

We split its proof in two steps. We will show:

\medskip

\noindent
{\bf Statement I:~} ${\mathcal F}$ includes a $k$-sunflower if 
$
|{\mathcal F}| \ge \lp \sqrt{10} - 2 \rp^2 \lp \frac{k}{\sqrt{10}-2}\rp^s s!
$
for any positive integers $s$ and $k$. 

\medskip

\noindent
{\bf Statement II:~} ${\mathcal F}$ includes a $k$-sunflower if 
$
|{\mathcal F}| \ge \lbr \frac{c k}{\log \min \lp s, k \rp} \rbr^s s!
$
for some constant $c$ and any integers $s \ge 2$ and $k \ge 2$.

\medskip

\noindent
It is clear that the two statements mean \refth{Main}. The rest of the paper is dedicated to the description of their proofs.

\section{Terminology and Related Facts} \label{Settings}

Denote an arbitrary set by $S$ that is a subset of $X$. 
Given a family ${\mathcal F}$ of sets of cardinality at most $s$, define
\(
&&
{\mathcal F}(S) \stackrel{def}{=} \lb U~:~ U \in {\mathcal F}
\textrm{~and~}
U \cap S \ne \emptyset 
\rb,
\\ &&
{\mathcal F}_j(S) \stackrel{def}{=} \lb U~:~ U \in {\mathcal F} 
\textrm{~and~}
|U \cap S| =j
\rb
\textrm{~~for positive integer $j$,}
\\ &&
{\mathcal F}_{sup}(S) \stackrel{def}{=} \lb U~:~ U \in {\mathcal F}
\textrm{~and~}
U \supset S 
\rb,
\textrm{~~and}
\\ &&
{\mathcal P}(S) \stackrel{def}{=} \lb (v, U) ~:~ v\in X,~
U \in {\mathcal F} 
\textrm{~and~} v \in U \cap S
\rb.
\)

Let $\ep \in (0, 1/8)$ be a constant and $k \in \Z^+$. Given such numbers, we use the following two functions as lower bounds on $|{\mathcal F}|$:
\(
&&
\Phi_1(s) \stackrel{def}{=} \lp \sqrt{10} - 2 \rp^2 \lp \frac{k}{\sqrt{10}-2} \rp^s s!, 
\eqand
\\ &&
\Phi_2(s) \stackrel{def}{=} 
\frac{k^s s!}{p_1 p_2 \cdots p_s},
\textrm{~~~where~~~}
p_j \stackrel{def}{=} 
\lb
\begin{array}{cc}
\ep \ln \min \lp j, k \rp & \textrm{if $j \ge 2$,}\\
\ep & \textrm{if $j=1$.}
\end{array}
\right.
\)
Here $\ln \cdot$ denotes the natural logarithm of a positive real number. We regard
$
\Phi_i(j) = 0
$
if $j \not\in \Z^+$ for each $i=1,2$.

The following lemma shows that Statement II is proved if $|{\mathcal F}| \ge \Phi_2(s)$ means a $k$-sunflower in ${\mathcal F}$.

\begin{lemma} \label{Phi2Bound}
There exists a constant $c$ such that 
$
\Phi_2(s) \le \lp \frac{c k}{\ln \min \lp k, s \rp} \rp^s s!
$
for any $s \ge 2$ and $k \ge 2$.
\end{lemma}

\noindent
It is shown by $p_1 p_2 \cdots p_s \ge \lp c' \ln \min \lp k, s \rp \rp^s$ for another constant $c'$. Its exact proof is found in Appendix.

We also have 
\beeqn && \label{Phi2}
\Phi_2 \lp s \rp = \frac{ks}{p_s} \Phi_2(s-1)= \frac{ks \cdot k(s-1)}{p_s p_{s-1}}\Phi_2(s-2) =\cdots
\nexteqline \nonumber
\frac{k^j s(s-1) \cdots (s-j+1)} {p_s p_{s-1} \cdots p_{s-j+1}} \Phi_2(s-j)
\nexteqline  \nonumber
\frac{k^j}{p_s p_{s-1} \cdots p_{s-j+1}}
\cdot \frac{s!}{(s-j)!}  \Phi_2(s-j)
\\ &\ge& \nonumber
\lp \frac{k}{p_s} \rp^j \frac{s!}{(s-j)!} \Phi_2(s-j),
\eeqn
for each positive integer $j <s$. The last inequality is due to $p_2 \le p_3 \le \cdots \le p_s$.

To derive another inequality from \refeq{Phi2}, we use Stirling's approximation  \\$\lim_{n \rightarrow \infty} \frac{n!}{\sqrt{2 \pi n}\lp \frac{n}{e} \rp^n} = 1$ where $e=2.71828...$ denotes the natural logarithm base. In a form of double inequality, it is known as
\[
\sqrt{2 \pi n} \cdot n^n e^{-n + \frac{1}{12n+1} } 
< n! <
\sqrt{2 \pi n} \cdot n^n e^{-n + \frac{1}{12n}},
\]
for $n \in \Z^+$ \cite{R55}. This means
\beeq{Stirling}
\sqrt{2 \pi n} \cdot n^n e^{-n}< n! 
\le \sqrt{n} \cdot n^n e^{-n+1}.
\eeq
Thus, 
\beeq{asymptotic}
{n \choose m} < \frac{n^{m}}{m!} < \exp \lp m \ln \frac{n}{m} + m  - \ln \sqrt{2 \pi m} \rp
< \exp \lp m \ln \frac{n}{m} + m \rp,
\eeq
for positive integers $n$ and $m$ such that $n \ge m$. We substitute 
$s! > \sqrt{2 \pi s} \cdot s^s e^{-s}$ and $(s-j)! \le \sqrt{s-j} \cdot (s-j)^{s-j} e^{-s+j+1}$ 
from \refeq{Stirling} into \refeq{Phi2} to see:

\begin{lemma} \label{Stirling2}
$\Phi_2(s) > \Phi_2(s-j) \exp \lp j \ln \frac{ks}{p_s}  - \frac{j^2}{s}  -1 \rp$
for a positive integer $j<s$. 
\end{lemma}

\noindent
A precise proof is also given in Appendix.

\section{The Improvement Method} \label{Method}

We will show Statements I and II by improving the original proof of the sunflower lemma in \cite{OriginalSF}. We review it in a way to introduce our proof method easily.  
The original proof shows by induction on $s$ that ${\mathcal F}$ includes a $k$-sunflower if $|{\mathcal F}|>\Phi_0(s)$, where 
\[
\Phi_0 (s) \stackrel{def}{=} (k-1)^s s!.
\]
The claim is clearly true in the induction basis $s=1$; the family ${\mathcal F}$ consists of more than $k-1$ different sets of cardinality at most 1 including a desired sunflower. 

To show the induction step, let
\[
{\mathcal B} = \lb B_1, B_2, \ldots, B_r \rb
\]
be a sub-family of ${\mathcal F}$ consisting of pairwise disjoint sets $B_i$, whose cardinality $r$ is maximum. We say that such ${\mathcal B}$ is a {\em maximal coreless sunflower in ${\mathcal F}$} for notational convenience in this paper. Also write
\[
B \stackrel{def}{=}  B_1 \cup B_2 \cup \cdots \cup B_r.
\]

We show $r \ge k$ to complete the induction step. We have 
\beeq{IH0}
\left| {\mathcal F}_{sup} \lp S \rp \right| \le \Phi_0(s - |S|) 
\textrm{~~for any $S \subset X$ such that $1 \le |S| < s$}, 
\eeq
for ${\mathcal F}_{sup} \lp S \rp$, the sub-family of ${\mathcal F}$ consisting of the sets including $S$ as defined in \refsec{Settings}. 
Otherwise a $k$-sunflower exists in ${\mathcal F}$ by induction hypothesis. Then the sub-family ${\mathcal F}(B)$ consisting of the sets intersecting with $B$ meets 
$
\left| {\mathcal F} \lp B \rp \right| \le |B| \Phi_0(s-1);
$
because for every element $v$ in $B$, the cardinality of ${\mathcal F}\lp \lb v \rb \rp= {\mathcal F}_{sup} \lp \lb v \rb \rp$ is bounded by $\Phi_0(s-1)$. Also $\Phi_0 \lp s- 1\rp= \frac{\Phi_0(s)}{ks}$. Thus, 
\beeqn
&& \label{SFProof}
\left| {\mathcal F} \lp B \rp \right|  \le |B| \Phi_0(s-1) \le rs \Phi_0(s-1)
\\ &=& \nonumber
rs \frac{\Phi_0(s)}{ks} = \frac{r}{k} \Phi_0(s) 
< \frac{r}{k} |{\mathcal F}|. 
\eeqn
Now $\left| {\mathcal F}(B) \right|$ is less than $|{\mathcal F}|$ unless $r \ge k$. In other words, if $r<k$, then ${\mathcal F}$ would have a set disjoint with any $B_i \in {\mathcal B}$, contradicting the maximality of $r=|{\mathcal B}|$. Hence $r \ge k$, meaning ${\mathcal B}$ includes a $k$-sunflower with an empty core. This proves the induction step.

\medskip

To improve this argument, we note that the proof works even if ${\mathcal F}_{sup} \lp \lb v \rb \rp$ for all $v \in B$ are disjoint. If so, each ${\mathcal F}_{sup} \lp \lb v \rb \rp$ includes elements $U \in {\mathcal F}$ only intersecting with $\lb v \rb$, and disjoint with $B- \lb v \rb$. Let $v \in B_i \in {\mathcal B}$. If we replace $B_i$ by any set $U$ in ${\mathcal F}_{sup} \lp \lb v \rb \rp$ such that $U \cap \lp B- \lb v \rb \rp= \emptyset$, then ${\mathcal B}$ is still a maximal coreless sunflower in ${\mathcal F}$.

On the other hand, if there are sufficiently many $U \in {\mathcal F}_{sup} \lp \lb v \rb \rp$ disjoint with $B - \lb v \rb$, we can find $U$ among them such that $|U \cap B_i|$ is much smaller than $s$. (Here we assume both $k$ are $s$ are are large enough.) This gives us the following contradiction: Due to the maximality of $r = |{\mathcal B}|$, the family ${\mathcal F}(U)$ must contain ${\mathcal F}(B_i) - {\mathcal F}(B- B_i)$. So ${\mathcal F}(U) \cap {\mathcal F}(B_i)$ includes most sets in ${\mathcal F}_{sup} \lp \lb v \rb \rp$ being a not too small family. But by \refeq{IH0}, $\left| {\mathcal F}(U) \cap {\mathcal F}(B_i) \right|$ is upper-bounded by 
\(
&&
s^2 \Phi_0(s-2)+ |U \cap B_i| \Phi_0(s-1)
\nexteqline 
s^2 \cdot \frac{\Phi_0(s)}{k^2 s (s-1)} 
+
|U \cap B_i| \cdot \frac{\Phi_0(s)}{ks}
\\ &<&
|{\mathcal F}|\lp \frac{1}{k^2\lp 1- \frac{1}{s} \rp} + \frac{|U \cap B_i|}{ks} \rp.
\)
As $|U \cap B_i|$ is much smaller than $s$, the cardinality $\left| {\mathcal F}(U) \cap {\mathcal F}(B_i) \right|$ is also small, $i.e.$ bounded by $|{\mathcal F}|$ times $O\lp \frac{1}{k^2} + \frac{|U \cap B_i|}{ks} \rp$. This contradiction on $\left| {\mathcal F}(U) \cap {\mathcal F}(B_i) \right|$ essentially means that if $r$ is around $k$, we can construct a larger coreless sunflower in ${\mathcal F}$. Hence $r$ must be more than $k$ with the cardinality lower bound $\Phi_0(s)$.

Our proof of \refth{Main} in the next section generalizes the above observation. By finding $B_i \in {\mathcal B}$ with sufficiently large $|{\mathcal F}(B_i) - {\mathcal B}(B- B_i)|$, we will show Statement I that  ${\mathcal F}$ such that $|{\mathcal F}| \ge \lp \sqrt{10} - 2 \rp^2 \lp \frac{k}{\sqrt{10}-2} \rp^2 s!$ includes a $k$-sunflower.

We will further extend this argument to show Statement II. Instead of finding just one such $B_i \in {\mathcal B}$, we will find  ${\mathcal B'} \subset {\mathcal B}$ such that a sub-family ${\mathcal H}$ of 
\beeq{FindB'}
{\mathcal F} \lp \bigcup_{S \in {\mathcal B}'} S \rp - {\mathcal F} \lp \bigcup_{S' \in {\mathcal B} - {\mathcal B}'} S' \rp 
\eeq
is sufficiently large. Then we show a maximal coreless sunflower in ${\mathcal H}$ whose cardinality is larger than $|{\mathcal B}'|$. This again contradicts the maximality of $r= |{\mathcal B}|$ to prove the second statement.

\section{Proof of \refth{Main}}

\subsection{Statement I} \label{Start}

We prove Statement I in this section. Put
\[
\delta = \sqrt{10}-3=0.16227\ldots,
\eqand 
x = \frac{k}{1+ \delta}  = \frac{k}{\sqrt{10}-2},
\]
then
\[
\Phi_1(s) = \lp \sqrt{10}-2 \rp^2 \lp \frac{k}{\sqrt{10}-2} \rp^s s!= 
\lp 1+ \delta \rp^2 x^s s!,
\]
as defined in \refsec{Settings}. We show that $|{\mathcal F}|\ge \Phi_1(s)$ means a $\lp 1+ \delta \rp x$-sunflower included in ${\mathcal F}$.

We prove it by induction on $s$. Its basis occurs when $s \le 2$. The claim is true by the sunflower lemma, since 
$
\Phi_1(s) = \lp 1+ \delta \rp^2 x^s s! \ge \lp (1+\delta) x \rp^s s!
= k^s s!>(k-1)^2 s!
$
if $s \le 2$. Assume true for $1, 2, \ldots, s-1$ and prove true for $s \ge 3$. As in \refsec{Method}, let  ${\mathcal B}=\lb B_1, B_2, \ldots, B_r \rb$ be a maximal coreless sunflower of cardinality $r$ in ${\mathcal F}$, and $B= B_1 \cup B_2 \cup \cdots \cup B_r$. Contrarily to the claim, let us assume 
\beeq{Cond1}
r < (1+\delta ) x.
\eeq
We will find a contradiction caused by \refeq{Cond1}.

Observe the following facts.

\begin{enumerate} [$\bullet$]
\item For any nonempty set $S \subset X$ such that $|S|<s$, if $|{\mathcal F}_{sup}(S)|\ge \Phi_1\lp s- |S| \rp$, the family ${\mathcal F}_{sup}(S)$ contains a $(1+\delta) x$-sunflower by induction hypothesis. Thus we assume 
\beeq{IH}
\left| {\mathcal F}_{sup}(S) \right| < \Phi_1\lp s- |S| \rp 
\textrm{~~for $S \subset X$ such that $1 \le |S|<s$}. 
\eeq
\item We also have $k \ge 3$, because  
$
\Phi_1(s) >0
$ 
for $k=1$, and\\
$
\Phi_2(s) = \lp \sqrt{10} - 2 \rp^2 \lp \frac{2}{\sqrt{10}-2} \rp^s s!
> s! = (k-1)^s s!
$
if $k = 2$. 
\item Since $r$ has the maximum value, 
\beeq{Cover}
{\mathcal F} = {\mathcal F}(B),
\eeq
$i.e.$, every set in ${\mathcal F}$ intersects with $B$. 
\item ${\mathcal P}(B)$, defined in \refsec{Settings} as the family of pairs $(v, U)$ such that $U \in {\mathcal F}$ and $v \in U \cap B$, has a cardinality bounded by 
\beeq{Bound1}
\left| {\mathcal P}(B) \right| \le |B| \Phi_1(s-1) \le rs \Phi_1(s-1) <(1+\delta)xs \cdot \frac{\Phi_1(s)}{xs}
\le (1+\delta) |{\mathcal F}|,
\eeq
due to \refeq{Cond1} and \refeq{IH}.
Here $\left| {\mathcal P}(B) \right| \le |B| \Phi_1(s-1)$ because for each $v \in B$, there are at most $\Phi_1(s-1)$ pairs $(v, U) \in {\mathcal P}(B)$. 
\end{enumerate}

We first see that many $U \in {\mathcal F}$ intersect with $B$ by cardinality $1$, $i.e.$, $|{\mathcal F}_1(B)|$ is sufficiently large. Observe two lemmas.

\begin{lemma} \label{L1}
$\left| {\mathcal F}_1(B) \right| > (1- \delta) |{\mathcal F}|$. 
\end{lemma}
\begin{proof}
Let $\left| {\mathcal F}_1(B) \right|  = (1- \delta') |{\mathcal F}|$ for some $\delta' \in [0,1]$. By \refeq{Cover}, there are $\delta' |{\mathcal F}|$ elements $U \in {\mathcal F}$ such that $|U \cap B| \ge 2$, each of which creates two  or more pairs in ${\mathcal P}(B)$. If $\delta' \ge \delta$, 
\[
|{\mathcal P}(B)|  \ge (1-\delta')|{\mathcal F}| + 2 \delta' |{\mathcal F}| =(1+ \delta')|{\mathcal F}|
\ge (1+ \delta) |{\mathcal F}|, 
\]
contradicting \refeq{Bound1}. Thus $\delta'<\delta$ proving the lemma. 
\end{proof}

\begin{lemma} \label{L2}
There exists $B_i \in \lb B_1, B_2, \ldots, B_r \rb$ such that $\left| {\mathcal F}_1 \lp B_i \rp  - {\mathcal F}\lp B - B_i \rp \right| \ge 
\frac{1- \delta}{1+\delta} \cdot \frac{\Phi_1(s)}{x}$.
\end{lemma}
\begin{proof}
By \reflm{L1}, there exists $B_i \in \lb B_1, B_2, \ldots, B_r \rb$ such that the number of $U \in {\mathcal F}$ intersecting with $B_i$ by cardinality 1, and disjoint with $B- B_i$, is at least 
\[
\frac{\lp 1- \delta \rp|{\mathcal F}|}{r} > \frac{1-\delta}{(1+\delta)x} \cdot |{\mathcal F}| 
\ge \frac{1- \delta}{1+\delta} \cdot \frac{\Phi_1(s)}{x}.
\]
The family of such $U$ is exactly ${\mathcal F}_1(B_i)- {\mathcal F}\lp B - B_i \rp$, so its cardinality is no less than $\frac{1- \delta}{1+\delta} \cdot \frac{\Phi_1(s)}{x}$. 
\end{proof}

Assume such $B_i$ is $B_1$ without loss of generality. Then
\[
\left| {\mathcal F} \lp B_1 \rp - {\mathcal F}\lp B - B_1 \rp \right| \ge \left| {\mathcal F}_1 \lp B_1 \rp - {\mathcal F}\lp B - B_1 \rp \right| \ge \frac{1- \delta}{1+\delta} \cdot \frac{\Phi_1(s)}{x}>0.
\]
We choose any element $B_1' \in {\mathcal F}_1 \lp B_1 \rp - {\mathcal F}\lp B - B_1 \rp$ that is not $B_1$. Switch $B_1$ with $B_1'$ in ${\mathcal B}$. 
Since $B_1'$ is disjoint with any of $B_2, B_3, \ldots, B_r$, the obtained family $\lb B'_1, B_2, \ldots, B_r \rb$ is another maximal coreless sunflower in ${\mathcal F}$. We see the following inequality. 

\newpage
-
\begin{lemma} \label{L3} 
\[
\left| {\mathcal F} (B_1) \cap {\mathcal F} (B'_1) \right| < 
\frac{\Phi_1(s)}{x} \lp \frac{1}{s} + \frac{1}{x} \rp. 
\]
\end{lemma}
\begin{proof}
A set $U \in F(B_1) \cap F(B_1')$ intersects with $B_1 \cap B_1'$ of cardinality 1, or both $B_1-B_1'$ and $B_1' - B_1$ of cardinality $s-1$. By \refeq{IH}, there are at most 
\(
&&
\Phi_1(s-1) + (s-1)^2 \Phi_1(s-2)
= \frac{\Phi_1(s)}{sx} + (s-1)^2 \frac{\Phi_1(s)}{s(s-1)x^2}
\\ &<&
\frac{\Phi_1(s)}{x} \lp \frac{1}{s} + \frac{1}{x} \rp
\)
such $U \in F(B_1) \cap F(B_1')$. The lemma follows.
\end{proof}

${\mathcal F}= {\mathcal F}(B_1' \cup B_2 \cup B_3 \cup \cdots \cup B_r)$ would be true if the cardinality $r$ of the new coreless sunflower $\lb B_1', B_2, B_3 , \ldots, B_r \rb$ were maximum. However, it means that every element in ${\mathcal F} \lp B_1 \rp - {\mathcal F}\lp B - B_1 \rp$ is included in ${\mathcal F}(B_1')$. Thus, ${\mathcal F} (B_1) \cap F (B'_1) \supset {\mathcal F} \lp B_1 \rp - {\mathcal F}\lp B - B_1 \rp$, leading to
\[
\left| {\mathcal F} (B_1) \cap F (B'_1) \right| \ge \left| {\mathcal F} \lp B_1 \rp - {\mathcal F}\lp B - B_1 \rp \right| \ge \frac{1- \delta}{1+\delta} \cdot \frac{\Phi_1(s)}{x}
\ge \frac{\Phi_1(s)}{x} \lp \frac{1}{s} + \frac{1}{x} \rp.
\]
Its last inequality is confirmed with $s \ge 3$, $k \ge 3$, $x= \frac{k}{1+\delta}$, and $\frac{1-\delta}{1+\delta} =\frac{1}{3} + \frac{1+\delta}{3}$ as $\delta = \sqrt{10}-3$. 
This contradicts \reflm{L3}, completing the proof of Statement I.

\subsection{Statement II}

We prove the second statement by further developing the above method. We show that ${\mathcal F}$ includes a $k$-sunflower if $|{\mathcal F}|> \Phi_2(s)$ for a choice of sufficiently small constant $\ep \in (0,1/8)$. This suffices to prove Statement II thanks to \reflm{Phi2Bound}. The proof is by induction on $s$. Its basis occurs when $\min (k, s)$ is smaller than a sufficiently large constant $c_1$, $i.e.$, when 
$
\min (k, s)
$
is smaller than a lower bound $c_1$ on $\min (k, s)$ required by the proof below. To meet this case, we choose $\ep \in (0,1/8)$ to be smaller than $1 \big/ \ln c_1$. Then $\Phi_2(s) \ge k^s s!>(k-1)^s s!$ by the definition of $\Phi_2$ in \refsec{Settings}. The family ${\mathcal F}$ thus includes a $k$-sunflower by the sunflower lemma in the basis. 

Assume true for $1, 2, \ldots, s-1$ and prove true for $s$. The two integers $k$ and $s$ satisfy
\beeq{SffLarge}
\min (k, s) \ge c_1,
\eeq
$i.e.$, they are sufficiently large. Put
\[
p  \stackrel{def}{=} \frac{1}{8} \ln \min (k, s),
~\eqand~~
x \stackrel{def}{=} \frac{k}{p}.
\]
$p$ is sufficiently large since both $k$ and $s$ are. 

\medskip

\noindent
{\bf Note.~}
The lower bound $c_1$ on $\min (s, k)$ is required in order to satisfy \refeq{c1_1}, \refeq{c1_2}, \refeq{c1_3} and \refeq{c1_4} below, which are inequalities with fixed coefficients and no $\ep$. 
We choose $c_1$ as the minimum positive integer such that $\min(k, s) \ge c_1$ satisfies the inequalities, and also $\ep$ as $\min \lp \frac{1}{2\ln c_1}, \frac{1}{9} \rp$.

\medskip

As $p_s =  \ep \ln \min (k, s) < p$ in \refeq{Phi2}, 
\beeq{IH4}
\Phi_2(s-1) = \frac{p_s}{ks} \Phi_2(s) 
<\frac{p}{ks} \Phi_2(s) 
= \frac{\Phi_2(s)}{xs}.
\eeq
Similarly to \refeq{IH}, we have
\[
|{\mathcal F}_{sup} \lp s- |S| \rp| < \Phi_2(s- |S|)
\textrm{~~for any $S \subset X$ with $1 \le |S|<s$},
\]
as induction hypothesis. We also keep denoting a maximal coreless sunflower in ${\mathcal F}$ by ${\mathcal B} = \lb B_1, B_2, \ldots, B_r \rb$, and $B_1 \cup B_2 \cup \cdots \cup B_r$ by $B$. 
Also \refeq{Cover} holds due to the maximality of $r=|{\mathcal B}|$.

We prove $r \ge k$ for the induction step. Suppose contrarily that
\beeq{Cond2}
x \le r < k,
\eeq
and we will find a contradiction. Here $r \ge x$ is confirmed similarly to \refeq{SFProof} in \refsec{Method}, $i.e.$, by
\[
|{\mathcal F}(B)| \le |B| \Phi_2(s-1) \le rs \Phi_2(s-1)
\le rs \frac{\Phi_2(s)}{xs} \le \frac{r}{x} |{\mathcal F}|,
\]
with \refeq{IH4}, so $r<x$ would contradict the maximality of $r=|{\mathcal B}|$.

We start our proof by showing a claim seen similarly to \reflm{L1}.

\begin{lemma}  \label{A1} There exists a positive integer $j \le 2p$ such that 
$\left| {\mathcal F}_j \lp B \rp \right| \ge \frac{\Phi_2(s)}{4p}$. 
\end{lemma}
\begin{proof}
We first show $\left| \sum_{0 \le j \le 2p} {\mathcal F}_j \lp B \rp  \right| \ge \frac{1}{2} |{\mathcal F}|$.
If not, there would be at least $\frac{1}{2} |{\mathcal F}|$ sets $U \in {\mathcal F}$ such that $|U \cap B| > 2p$ by the definition of ${\mathcal F}_j$ given in \refsec{Settings}. Each such $U$ creates at least $\lc 2p \rc$ pairs $(v, U) \in {\mathcal P}\lp B \rp$, so
\[
\left| {\mathcal P}\lp B \rp \right| \ge 2p  \cdot \frac{|{\mathcal F}|}{2} = p |{\mathcal F}|. 
\]
However, similarly to \refeq{Bound1},
\[
\left| {\mathcal P}(B) \right| \le |B| \Phi_2(s-1) \le r s \Phi_2(s-1) <ks  \cdot \frac{p \Phi_2(s)}{ks}
\le p |{\mathcal F}|,
\]
by induction hypothesis and \refeq{IH4}. By the contradictory two inequalities,\\ $\left| \sum_{0 \le j \le 2p} {\mathcal F}_j \lp B \rp  \right| < \frac{1}{2} |{\mathcal F}|$ is false. Thus $\left| \sum_{0 \le j \le 2p} {\mathcal F}_j \lp B \rp  \right| \ge \frac{1}{2} |{\mathcal F}|$. 

Let $j$ be an integer in $[1, 2p]$ with the maximum cardinality of ${\mathcal F}_j \lp B \rp $. By the above claim and \refeq{Cover}, 
$
\left| {\mathcal F}_j \lp B \rp \right| \ge \frac{|{\mathcal F}|}{2} \cdot \frac{1}{2p} \ge \frac{\Phi_2(s)}{4p}
$,
proving the lemma.
\end{proof}

Fix this integer $j \in [1, 2p]$. Next we find a small sub-family ${\mathcal B}'$ of ${\mathcal B}$ such that \refeq{FindB'} is large enough. For each non-empty ${\mathcal B}' \subset {\mathcal B}$, define
\(
&&
{\mathcal G} \lp {\mathcal B}' \rp \stackrel{def}{=} 
\lb U ~:~ U \in {\mathcal F}_j \lp B \rp,~
\forall S \in {\mathcal B}', U \cap S \ne \emptyset
\textrm{~and~} 
\forall S' \in {\mathcal B} - {\mathcal B}', U \cap S' = \emptyset
\rb.
\)
Observe a lemma regarding ${\mathcal G}\lp {\mathcal B}' \rp$.

\begin{lemma} \label{A2}
There exists a nonempty sub-family ${\mathcal B}' \subset {\mathcal B}$ of cardinality at most $j$ such that 
$
\left| {\mathcal G}\lp {\mathcal B}' \rp \right| \ge \frac{\Phi_2(s)}{8p {r \choose j}}$.
\end{lemma}
\begin{proof}
By definition, the cardinality of ${\mathcal B}'$ such that ${\mathcal G}\lp {\mathcal B}' \rp \ne \emptyset$ does not exceed $j$. Thus there are at most 
\(
&&
{r \choose j} + {r \choose j-1} + {r \choose j-2} + \cdots + {r \choose 1}
\\ &<&
{r \choose j} \lp 1 + \frac{j}{r-j+1} + \lp \frac{j}{r-j+1} \rp^2 + \lp \frac{j}{r-j+1} \rp^3 + \cdots \rp
\le
2 {r \choose j}
\)
such possible ${\mathcal B}' \subset {\mathcal B}$. Here its truth is confirmed by the following arguments.
\begin{enumerate} [$\bullet$]
\item ${n \choose m-1} = \frac{m}{n-m+1}{n \choose m}$ for any $n, m \in \Z^+$ such that $m \le n$. So ${r \choose j-1} = \frac{j}{r-j+1}{r \choose j}$, ${r \choose j-2} = \frac{j-1}{r-j+2}{r \choose j-1}<\lp \frac{j}{r-j+1} \rp^2 {r \choose j}$, 
${r \choose j-3} = \frac{j-2}{r-j+3}{r \choose j-2}<\lp \frac{j}{r-j+1} \rp^3 {r \choose j}$, $\cdots$.
\item The last inequality is due to 
$
r \ge x =\frac{k}{p} = \frac{8k}{\ln \min \lp k, s \rp} > \frac{k}{\ln k}
$
by \refeq{Cond2}, and $j \le 2p < \ln k$. Thus 
\beeq{c1_1}
\frac{j}{r-j+1} < \frac{\ln k}{\frac{k}{\ln k} -  \ln k+1}< \frac{1}{2},
\eeq
by \refeq{SffLarge} where $k$ is sufficiently large. So the last inequality holds in the above. 
\end{enumerate}

Then by \reflm{A1}, there exists at least one nonempty ${\mathcal B}' \subset {\mathcal B}$ such that 
\[
\left| {\mathcal G}\lp {\mathcal B}' \rp \right| 
> \frac{|{\mathcal F}_j(B)|}{2{r \choose j}}
\ge \frac{\Phi_2(s)}{2 {r \choose j} \cdot 4p} 
=\frac{\Phi_2(s)}{8p {r \choose j}}.
\]
The lemma follows. 
\end{proof}

We now construct a sub-family ${\mathcal H}$ of \refeq{FindB'} in which we will find a larger maximal coreless sunflower. Fix a sub-family ${\mathcal B}' \subset {\mathcal B}$ decided by \reflm{A2}. Put
\(
&&
B' \stackrel{def}{=} \bigcup_{S \in {\mathcal B}'} S,
\\ &&
r' \stackrel{def}{=} |{\mathcal B}'| \le j \le 2p, 
\\ \textrm{and~~} &&
{\mathcal H} \stackrel{def}{=} F_j \lp B'  \rp - F \lp B - B' \rp
=
{\mathcal F}_j \lp \bigcup_{S \in {\mathcal B}'} S \rp - {\mathcal F} \lp \bigcup_{S' \in {\mathcal B} - {\mathcal B}'} S' \rp.
\)
The family ${\mathcal H}$ includes ${\mathcal G}\lp {\mathcal B'} \rp$ by definition, so 
\beeq{CardinalityH}
|{\mathcal H}| \ge \frac{\Phi_2(s)}{8p {r \choose j}},
\eeq
by \reflm{A2}. If we find a maximal coreless sunflower in ${\mathcal H}$ whose cardinality is larger than $r' = |{\mathcal B}'|$, it means the existence of a coreless sunflower in ${\mathcal F}$ with cardinality larger than $r$, since any $U \in {\mathcal H}$ is disjoint with $B-B'$.

Extending the notation ${\mathcal F}(S)$, write
\[
{\mathcal H}\lp S \rp \stackrel{def}{=} \lb U~:~ U \in {\mathcal H} 
\textrm{~and~}
U \cap S \ne \emptyset
\rb,
\]
for a nonempty set $S \subset X$. Then ${\mathcal H}(\lb v \rb)$ for an element $v \in X$ is the family of $U \in {\mathcal H} \subset {\mathcal F}$ containing $v$.

Let us show two lemmas on ${\mathcal H}$ and ${\mathcal H}\lp \lb v \rb \rp$.  By them we will see that the latter is sufficiently smaller than the former. 

\begin{lemma} \label{A2'}
$\left| {\mathcal H}  \right| > s^j \Phi_2(s-j) \exp \lp - 4p  \rp$.
\end{lemma}
\begin{proof}
We have two facts on \refeq{CardinalityH}. 
\begin{enumerate} [$\bullet$]
\item The natural logarithm of the denominator $8p {r \choose j}$ is upper-bounded by 
\[
\ln 8p {r \choose j}
< j \ln \frac{r}{j} + j + \ln 8p
< j \ln \frac{k}{j} + j + \ln 8p
= j \ln \frac{xp}{j} + j + \ln 8p,
\]
due to \refeq{asymptotic}, \refeq{Cond2} and $x = \frac{k}{p}$. 
\item Let $d = \ln \Phi_2(s) - \ln \Phi_2(s-j)$, or $\Phi_2(s) =  \Phi_2(s-j) \exp(d)$. It satisfies
\[
d> j \ln \frac{ks}{p_s} -\frac{j^2}{s} -1
>
j \ln \frac{ks}{p} -\frac{j^2}{s} -1
= j \ln sx -\frac{j^2}{s} -1
> j \ln sx -j -1,
\]
by \reflm{Stirling2}, $p_s < p$, and $j \le s$.
\end{enumerate}
Then, 
\(
\left| {\mathcal H} \right| 
&\ge&\frac{\Phi_2(s)}{8p{r \choose j}}
>
\Phi_2(s-j) \exp \lp 
\lp j \ln s x  - j  -1  \rp
- \lp j \ln \frac{xp}{j} + j + \ln 8p \rp
\rp
\nexteqline
s^j \Phi_2(s-j) \exp \lp 
 - j \ln \frac{p}{j}  - 2 j  - \ln 8p   - 1 
\rp.
\)

Find $\max_{1 \le j \le 2p} \lp j \ln \frac{p}{j}  + 2j \rp$ regarding $j$ as a real parameter. The maximum value $(4 - 2 \ln 2)p$ is achieved when $j=2p$. Since $p$ is sufficiently large by \refeq{SffLarge}, 
\beeq{c1_2}
2 p \ln 2>1+ \ln 8p.
\eeq
Hence, 
\(
\left| {\mathcal H}  \right| &>& 
s^j \exp \lp 
 - j \ln \frac{p}{j}  - 2 j  - \ln 8p   - 1 
\rp
\ge
s^j \exp \lp 
 - (4- 2 \ln 2) p  - \ln 8p   - 1 
\rp
\\ &>&
s^j \Phi_2(s-j) \exp \lp - 4 p \rp, 
\)
completing the proof. 
\end{proof}

\begin{lemma} \label{A3} 
The following two statements hold true.
\begin{enumerate} [i)]
\item $|{\mathcal H}\lp \lb v \rb \rp| \le \frac{1}{s} \cdot e^{7p} \cdot |{\mathcal H}|$ for every $v \in B'$. 
\item $|{\mathcal H}\lp \lb v \rb \rp| \le \frac{1}{(s-j)x} \cdot e^{7p} \cdot |{\mathcal H}|$ for every $v \in X- B'$. 
\end{enumerate}
\end{lemma}
\begin{proof}
i): Let $U$ be any set in ${\mathcal H}\lp \lb v \rb \rp$ for the given $v \in B'$, and write $U'= U \cap B'$. Since $U'$ has cardinality $j$ containing $v$, there are no more than 
\[
{|B'|-1 \choose j-1} \le {r' s-1 \choose j-1}= \frac{j}{r' s} {r' s \choose j} \le \frac{j}{s}{r' s \choose j}
\]
choices of $U'$.  Here the identity ${r' s \choose j}= \frac{r' s}{j} {r' s-1 \choose j-1}$ is used. By the induction hypothesis on $s$, the number of $U \in {\mathcal H}\lp \lb v \rb \rp$ is upper-bounded by
\(
\left| {\mathcal H}\lp \lb v \rb \rp \right| &\le& \frac{j}{s}{r' s \choose j}  \cdot  \Phi_2(s-j)
\le \frac{\Phi_2(s-j)}{s}
\exp\lp \ln j + \lp j \ln \frac{r' s}{j} + j \rp \rp
\\ &\le&
s^j  \Phi_2(s-j) \frac{\exp \lp j \ln \frac{r'}{j} + j + \ln j \rp}{s}
\le
s^j  \Phi_2(s-j) \frac{\exp \lp j + \ln j \rp}{s}
\\ &\le&
s^j  \Phi_2(s-j) \frac{\exp \lp 2p  + \ln 2p \rp}{s}
<
s^j  \Phi_2(s-j) \frac{\exp \lp 3p \rp}{s},
\)
where ${r' s \choose j} \le \exp \lp j \ln \frac{r' s}{j} + j \rp$ is due to \refeq{asymptotic}, and
\beeq{c1_3}
2p + \ln 2p < 3p,
\eeq
due to \refeq{SffLarge}.

With the previous lemma, we see that the ratio $|{\mathcal H}\lp \lb v \rb \rp|\big/ |{\mathcal H}|$ does not exceed
\(
&&
\frac{s^j  \Phi_2(s-j) \exp \lp 3p \rp \big/ s}
{s^j  \Phi_2(s-j) \exp \lp - 4p \rp}
= \frac{e^{7p}}{s},
\)
proving i). 

ii): As $v \in X - B'$, the number of choices of $U' =U \cap B'$ is now ${|B'| \choose j} \le {r' s \choose j}$. For each such $U'$, the number of sets $U \in {\mathcal H}$ containing $U' \cup \lb v \rb$ is at most
\(
\Phi_2(s-j-1) &=& \frac{p_{s-j}}{k(s-j)} \Phi_2(s-j)
<\frac{p}{k(s-j)} \Phi_2(s-j)
=\frac{1}{x(s-j)} \Phi_2(s-j),
\)
by induction hypothesis and $p_{s-j} \le p_s < p$. 
Hence $|{\mathcal H}\lp \lb v \rb \rp|$ is upper-bounded by $\frac{\Phi_2(s-j)}{(s-j)x}{r' s \choose j}$. Then argue similarly to i). 
\end{proof}

\reflm{A3} means that ${\mathcal H}$ includes a coreless sunflower of cardinality more than $r'=|{\mathcal B}'|$. Let us formally prove it with the following lemma.

\begin{lemma} \label{A4} 
$
|{\mathcal B}'| \cdot |{\mathcal H}(V)| \le \frac{1}{2} |{\mathcal H}|
$
for every $V \in {\mathcal H}$.
\end{lemma}
\begin{proof}
Let
\[
y \stackrel{def}{=} \min \lp s, k \rp,
\textrm{~~~so that~~~}
p = \frac{1}{8} \ln y,~~
y = e^{8p},~~
s \ge y,
\textrm{~~and~~}
x = \frac{k}{p} \ge \frac{y}{p}.
\]
We have
\beeq{c1_4}
p \ge 1
\eqand
8 p^3 e^{-p} < \frac{1}{2},
\eeq
due to \refeq{SffLarge}. 

Fix each $V \in {\mathcal H}$. By \reflm{A3}, the family ${\mathcal H}(V)$ has a cardinality bounded by 
\beeqn
|{\mathcal H}(V)| &=& \nonumber
\bigcup_{v \in V} |{\mathcal H}\lp \lb v \rb \rp|  \le \lp j \cdot \frac{e^{7p}}{s}
+ (s-j) \frac{e^{7p}}{(s-j)x} 
\rp |{\mathcal H}|
\\ &\le& \nonumber 
j e^{7p}\lp  \frac{1}{s}
+ \frac{1}{x}
\rp |{\mathcal H}|
\le 
j e^{7p}\lp  \frac{1}{y}
+ \frac{p}{y}
\rp |{\mathcal H}|
\le
j e^{7p} \cdot \frac{2p}{y} \cdot |{\mathcal H}|
\\ &\le& \nonumber 
4p^2 e^{7p} \cdot \frac{|{\mathcal H}|}{y}
= 
4p^2 e^{7p} \cdot \frac{|{\mathcal H}|}{e^{8p}}
=4p^2 e^{-p} \cdot |{\mathcal H}|.
\eeqn
Therefore, by $|{\mathcal B}'| = r' \le j \le 2p$, 
\[
|{\mathcal B}'| \cdot |{\mathcal H}(V)| 
\le 2p |{\mathcal H}(V)|
\le 8 p^3 e^{-p} \cdot |{\mathcal H}|
< \frac{1}{2} |{\mathcal H}|.
\]
The lemma follows.
\end{proof}

Hence each $V \in {\mathcal H}$ intersects with at most $\frac{|{\mathcal H}|}{2 |{\mathcal B}'|}$ sets in ${\mathcal H}$. There exist more than $|{\mathcal B}'|$ pairwise disjoint sets in ${\mathcal H}$. By definition, every element in ${\mathcal H}$ is disjoint with a set in ${\mathcal B}- {\mathcal B}'$. The cardinality $r$ of the coreless sunflower ${\mathcal B}$ in ${\mathcal F}$ is therefore not maximum. This contradiction proves Statement II, completing the proof of \refth{Main}.

\section*{Appendix: Proofs of Lemmas \ref{Phi2Bound} and \ref{Stirling2}}

{\bf Lemma \ref{Phi2Bound}.~}
There exists a constant $c>0$ such that 
$
\Phi_2(s) \le \lp \frac{c k}{\ln \min \lp k, s \rp} \rp^s s!
$
for any $s \ge 2$ and $k \ge 2$.
\begin{proof}
We first show the lemma when $s \le k$. By the definition of $\Phi_2(s)$ in \refsec{Settings}, 
\[
\Phi_2(s) 
= \frac{1}{\ln 2 \cdot \ln 3 \cdots \ln s} \lp \frac{k}{\ep} \rp^{s} s! .
\]
We assume $s \ge 3$ since the claim is trivially true for $s=2$. 
It suffices to show
$
\frac{1}{\ln 2 \cdot \ln 3 \cdots \ln s} \lp \frac{k}{\ep} \rp^{s} s!
\le \lp \frac{e^2}{\ln s}  \rp^s  \lp \frac{k}{\ep} \rp^{s} s!,
$
or
\beeqn
&& \label{eqPhi2Bound}
\ln 2 \cdot \ln 3 \cdots \ln s \ge \lp \frac{\ln s}{e^2} \rp^s
\\ &\Leftrightarrow~~~& \nonumber
\ln \ln 2 + \ln \ln 3 + \cdots + \ln \ln s \ge s \ln \ln s - 2s.
\eeqn
$\ln \ln x$ is a smooth, monotonically increasing function of $x \in \R^+$, so 
\(
&&
\ln \ln 3 + \ln \ln 4 + \cdots + \ln \ln s>
\int_{2}^{s} \ln \ln x ~dx 
\nexteqline (s \ln \ln s - li(s) )- (2 \ln \ln 2  - li(2)).
\)
Here $li(s)$ is the logarithmic integral $\int_{0}^s \frac{dx}{\ln x}$. As $\ln \ln 2<0$, 
\[
\ln \ln 2 + \ln \ln 3 + \cdots + \ln \ln s>
s \ln \ln s - li(s) + li(2),
\]
where $li(s) - li(2) = \int_{2}^s \frac{dx}{\ln x}$ is upper-bounded by $(s-2)/ \ln 2 < 2s$. Hence, 
$
\ln \ln 2 + \ln \ln 3 + \cdots + \ln \ln s \ge s \ln \ln s - 2s,
$
proving the lemma when $s \le k$.

If $s>k$, the lemma is also proved by \refeq{eqPhi2Bound}:
\(
\Phi_2(s) &=&
\frac{1}{\ln 2 \cdot \ln 3 \cdots \ln k \cdot \lp \ln k \rp^{s-k}} \lp \frac{k}{\ep} \rp^{s} s!
\\ &\le&
\frac{1}{\lp \frac{\ln k}{e^2} \rp^k \cdot \lp \ln k \rp^{s-k}} \lp \frac{k}{\ep} \rp^{s} s!
<
\lp \frac{e^2 k}{\ep \ln k} \rp^s s!.
\)
This completes the proof. 
\end{proof}

\medskip

{\bf Lemma \ref{Stirling2}.~}
$\Phi_2(s) > \Phi_2(s-j) \exp \lp j \ln \frac{ks}{p_s}  - \frac{j^2}{s}  -1 \rp$
for a positive integer $j<s$. 
\begin{proof}
It suffices to show  $\lp \frac{k}{p_s} \rp^j \frac{s!}{(s-j)!}> \exp \lp j \ln \frac{ks}{p_s}  - \frac{j^2}{s}  -1 \rp$ due to \refeq{Phi2}, or 
\[
\frac{s!}{(s-j)!}> \exp \lp j \ln s  - \frac{j^2}{s}  -1 \rp. 
\]
By \refeq{Stirling}, $\ln \frac{s!}{(s-j)!}$ is at least
\(
&&
\lp s \ln s - s + \ln \sqrt{2 \pi s} \rp
- \lp (s-j) \ln (s-j) - (s-j) + \ln \sqrt{s-j} + 1 \rp
\\ &=& 
s \ln s - (s-j) \ln (s-j) - j + \ln \frac{\sqrt{2 \pi s}}{\sqrt{s-j}} - 1
\\ &>& 
s \ln s - (s-j) \lp \ln s +   \ln \lp 1 - \frac{j}{s} \rp  \rp - j - 1
\nexteqline
j \ln s - \lp s-j \rp \ln \lp 1 - \frac{j}{s} \rp - j - 1. 
\)
By the Taylor series of natural logarithm, 
$
- \ln \lp 1- \frac{j}{s} \rp = \sum_{i=1}^\infty \frac{1}{i} \lp \frac{j}{s} \rp^i > \frac{j}{s}
$.
From these,
\[
\ln \frac{s!}{(s-k)!} 
> j \ln s + \lp s-j \rp \frac{j}{s} - j -1
\ge j \ln s - \frac{j^2}{s} -1,
\]
which is equivalent to the desired inequality to prove the lemma.
\end{proof}

\bibliographystyle{amsplain}

\bibliography{ref}

\end{document}